\documentclass{amsart}

\usepackage{curves}
\usepackage{epic}
\usepackage{eepic}
\usepackage{amsmath}
\usepackage{amscd}
\usepackage{amsfonts}
\usepackage{amssymb}
\usepackage{latexsym}
\usepackage{pifont}
\usepackage{array}
\usepackage{oldgerm}

\newtheorem{thm}{Theorem}[section]
\newtheorem{pro}[thm]{Proposition}
\newtheorem{lem}[thm]{Lemma}
\newtheorem{cor}[thm]{Corollary}
\newtheorem{thm&def}[thm]{Theorem \& Definition}
\newtheorem{lem&def}[thm]{Lemma \& Definition}

\theoremstyle{definition}
\newtheorem{defi}[thm]{Definition}                   

\theoremstyle{remark}
\newtheorem{rmk}[thm]{Remark}

\newcommand{\ncm}{\newcommand}

\ncm{\decorenum}[1]{

\begin{enumerate}}
\ncm{\decorenumend}{
\end{enumerate}

}

\ncm{\decorenumprime}[1]{

\begin{enumerate}}
\ncm{\decorenumprimeend}{
\end{enumerate}

}

\ncm{\Cat}{\mathsf{Cat}}
\ncm{\CAT}{\mathsf{CAT}}
\ncm{\ob}{\operatorname{ob}}
\ncm{\Nat}{\operatorname{Nat}}
\ncm{\Set}{\mathsf{Set}}
\ncm{\Ab}{\mathsf{Ab}}
\ncm{\Add}{\mathsf{Add}}
\ncm{\Mnd}{\mathsf{Mnd}}
\ncm{\Mon}{\mathsf{Mon}}
\ncm{\Fun}{\mathsf{Fun}}
\ncm{\MonCat}{\mathsf{MonCat}}
\ncm{\MonFun}{\mathsf{MonFun}}
\ncm{\Elt}{\mathsf{Elt}\,}
\ncm{\Sub}{\mathsf{Sub}\,}
\ncm{\Flat}{\mathsf{Flat}}
\ncm{\add}{\text{\Large\textsl{a}}}
\ncm{\proj}{\mathsf{proj}}
\ncm{\Col}{\mathsf{Col}\,}
\ncm{\fgmod}[1]{\mathsf{mod}\text{-}#1}
\ncm{\UEE}{\mathsf{u\text{-}eff}}
\ncm{\EEE}{\mathsf{e\text{-}eff}}
\ncm{\HEE}{\mathsf{h\text{-}eff}}
\ncm{\EE}{\mathsf{eff}}
\ncm{\Split}{\mathsf{split}}
\ncm{\Cl}{\mathsf{Cl}}
\ncm{\ord}{\mathsf{ord}}
\ncm{\Mod}[1]{\mathbf{Mod}\text{-}#1}
\ncm{\Bimod}[2]{#1\text{-}\mathbf{Mod}\text{-}#2}
\ncm{\Comod}[1]{\mathbf{Comod}\text{-}#1}
\ncm{\Sh}{\mathbf{Sh}}
\ncm{\Gr}{\mathbf{Gr}}

\ncm{\A}{\mathcal{A}}
\ncm{\B}{\mathcal{B}}
\ncm{\C}{\mathcal{C}}
\ncm{\D}{\mathcal{D}}
\ncm{\E}{\mathcal{E}}
\ncm{\F}{\mathcal{F}}
\ncm{\G}{\mathcal{G}}
\ncm{\Ha}{\mathcal{H}}
\ncm{\I}{\mathcal{I}}
\ncm{\J}{\mathcal{J}}
\ncm{\K}{\mathcal{K}}
\ncm{\Ll}{\mathcal{L}}
\ncm{\M}{\mathcal{M}}
\ncm{\N}{\mathcal{N}}
\ncm{\Ou}{\mathcal{O}}
\ncm{\Pee}{\mathcal{P}}
\ncm{\R}{\mathcal{R}}
\ncm{\X}{\mathcal{X}}
\ncm{\V}{\mathcal{V}}
\ncm{\U}{\mathcal{U}}
\ncm{\T}{\mathcal{T}}
\ncm{\Teven}{\mathcal{T}_\mathsf{e}}
\ncm{\Tcan}{\mathcal{T}_\mathsf{can}}

\ncm{\dom}{\operatorname{dom}}
\ncm{\cod}{\operatorname{cod}}
\ncm{\End}{\operatorname{End}}
\ncm{\Aut}{\operatorname{Aut}}
\ncm{\Hom}{\operatorname{Hom}}
\ncm{\kernel}{\operatorname{ker}}
\ncm{\Ker}{\operatorname{Ker}}
\ncm{\coker}{\operatorname{coker}}
\ncm{\Coker}{\operatorname{Coker}}
\ncm{\im}{\operatorname{im}}
\ncm{\Img}{\operatorname{Im}}
\ncm{\coim}{\operatorname{coim}}
\ncm{\id}{\operatorname{id}}
\ncm{\Center}{\operatorname{Center}}
\ncm{\colim}{\operatorname{colim}}
\ncm{\Colim}[1]{\underset{#1}{\operatorname{colim}}}
\ncm{\Lan}{\operatorname{Lan}}
\ncm{\Cone}{\operatorname{Cone}}
\ncm{\ev}{\operatorname{ev}}
\ncm{\coev}{\operatorname{coev}}
\ncm{\cf}{\operatorname{cf}}
\ncm{\hgt}{\operatorname{ht}}

\ncm{\ci}{\,\circ\,}
\ncm{\smp}{\ast}
\ncm{\smpc}{\circledast}
\ncm{\bu}{\bullet}
\ncm{\bo}{\,\Box\,}
\ncm{\ot}{\otimes}
\ncm{\oV}{\odot}
\ncm{\hot}{\,\bar{\ot}\,}
\ncm{\hotobj}{\,{\scriptstyle\bullet}\,}
\ncm{\votobj}{\circledcirc}               
\ncm{\vot}{\,{\ot}\,}
\ncm{\x}{\times}
\ncm{\subsetnoteq}{\underset{\neq}{\subset}}
\ncm{\ex}[1]{\underset{\scriptstyle #1}{\times}}
\ncm{\am}[1]{\underset{\scriptscriptstyle #1}{\ot}}
\ncm{\amo}[1]{\underset{\scriptstyle #1}{\ot}}
\ncm{\mash}{\Pisymbol{psy}{35}}
\ncm{\mashed}[1]{\underset{\scriptscriptstyle #1}{\Pisymbol{psy}{35}}}
\ncm{\cross}[1]{\underset{\scriptstyle #1}{\rtimes}}
\ncm{\rarr}[1]{\stackrel{#1}{\longrightarrow}}
\ncm{\larr}[1]{\stackrel{#1}{\longleftarrow}}
\ncm{\mapsot}{\leftarrow\!\!\!\raisebox{1pt}{$\scriptscriptstyle |$}}
\ncm{\oR}{\am{R}}
\ncm{\oE}{\am{E}}
\ncm{\oA}{\am{A}}
\ncm{\oS}{\am{S}}
\ncm{\oT}{\am{T}}
\ncm{\cop}{\Delta}
\ncm{\nulT}{^{(0)}}
\ncm{\oneT}{^{(1)}}
\ncm{\twoT}{^{(2)}}
\ncm{\threeT}{^{(3)}}
\ncm{\nulB}{_{(0)}}
\ncm{\oneB}{_{(1)}}
\ncm{\twoB}{_{(2)}}
\ncm{\threeB}{_{(3)}}
\ncm{\eps}{\varepsilon}
\ncm{\bra}{\langle}
\ncm{\ket}{\rangle}
\ncm{\dirim}[1]{{#1}_*}
\ncm{\invim}[1]{{#1}^*}

\ncm{\asso}{\mathbf{a}}
\ncm{\luni}{\mathbf{l}}
\ncm{\runi}{\mathbf{r}}
\ncm{\iso}{\stackrel{\sim}{\rightarrow}}
\ncm{\iiso}{\rarr{\sim}}
\ncm{\ract}{\,\triangleleft\,}
\ncm{\lact}{\triangleright}
\ncm{\under}{\mbox{\,\rm\_}\,}
\ncm{\adj}{\dashv}
\ncm{\adjoint}{\dashv}
\ncm{\into}{\hookrightarrow}

\ncm{\can}{\mathrm{can}}
\ncm{\fgp}{\mathrm{fgp}}
\ncm{\op}{\mathrm{op}}
\ncm{\reg}{\mathrm{reg}}
\ncm{\coop}{\mathrm{coop}}
\ncm{\rev}{\mathrm{rev}}
\ncm{\sst}{\scriptstyle}
\ncm{\ssst}{\scriptscriptstyle}

\ncm{\eqby}[1]{\stackrel{(\ref{#1})}{=}}
\ncm{\lef}{{\ssst <}}
\ncm{\righ}{{\ssst >}}

\ncm{\NN}{\mathbb{N}}
\ncm{\ZZ}{\mathbb{Z}}
\ncm{\QQ}{\mathbb{Q}}
\ncm{\GG}{\mathbf{G}}
\ncm{\FF}{\mathbb{F}}
\ncm{\TT}{\mathsf{T}}
\ncm{\Q}{\mathsf{Q}}

\ncm{\pisharp}{\Pisymbol{psy}{35}}

\ncm{\parallelpair}{
\parbox{43pt}{
\begin{picture}(43,8)
\put(3,6){\vector(1,0){37}}
\put(3,2){\vector(1,0){37}}
\end{picture}
}}
\ncm{\pair}[2]{\overset{#1}{\underset{#2}{\parallelpair}}}

\ncm{\longparallelpair}{
\parbox{63pt}{
\begin{picture}(63,8)
\put(3,6){\vector(1,0){57}}
\put(3,2){\vector(1,0){57}}
\end{picture}
}}
\ncm{\longpair}[2]{\overset{#1}{\underset{#2}{\longparallelpair}}}

\ncm{\longerparallelpair}{
\parbox{83pt}{
\begin{picture}(83,8)
\put(3,6){\vector(1,0){77}}
\put(3,2){\vector(1,0){77}}
\end{picture}
}}
\ncm{\longerpair}[2]{\overset{#1}{\underset{#2}{\longerparallelpair}}}

\ncm{\longrightarrowtail}{
\parbox{40pt}{
\begin{picture}(40,8)
\put(6,4){\line(-1,1){1.5}}
\put(6,4){\line(-1,-1){1.5}}
\put(6,4){\vector(1,0){31}}
\end{picture}
}}

\ncm{\longerrightarrowtail}{
\parbox{60pt}{
\begin{picture}(60,8)
\put(6,4){\line(-1,1){1.5}}
\put(6,4){\line(-1,-1){1.5}}
\put(6,4){\vector(1,0){51}}
\end{picture}
}}

\ncm{\longrarr}[1]{
\overset{#1}{
\parbox{40pt}{
\begin{picture}(40,8)
\put(3,4){\vector(1,0){34}}
\end{picture}
}}}

\ncm{\longerrarr}[1]{
\overset{#1}{
\parbox{70pt}{
\begin{picture}(70,8)
\put(3,4){\vector(1,0){64}}
\end{picture}
}}}

\ncm{\longerlarr}[1]{
\overset{#1}{
\parbox{70pt}{
\begin{picture}(70,8)
\put(67,4){\vector(-1,0){64}}
\end{picture}
}}}

\ncm{\longlarr}[1]{\overset{#1}{\longleftarrow}}

\ncm{\antiparallelpair}{
\parbox{23pt}{
\begin{picture}(23,4)
\put(3,3){\vector(1,0){17}}
\put(20,1){\vector(-1,0){17}}
\end{picture}
}}

\ncm{\invantiparallelpair}{
\parbox{23pt}{
\begin{picture}(23,4)
\put(3,1){\vector(1,0){17}}
\put(20,3){\vector(-1,0){17}}
\end{picture}
}}

\ncm{\dualpair}[2]{\overset{#1}{\underset{#2}{\antiparallelpair}}}

\ncm{\invdualpair}[2]{\overset{#1}{\underset{#2}{\invantiparallelpair}}}

\ncm{\coantiparallelpair}{
\parbox{23pt}{
\begin{picture}(23,4)
\put(3,1){\vector(1,0){17}}
\put(20,4){\vector(-1,0){17}}
\end{picture}
}}

\ncm{\codualpair}[2]{\overset{#1}{\underset{#2}{\coantiparallelpair}}}

\ncm{\binarydirectsum}[7]{#1\codualpair{#2}{#3}#4\dualpair{#5}{#6}#7}

\ncm{\equalizer}[1]{\overset{#1}{\longrightarrowtail}}

\ncm{\epi}[1]{\overset{#1}{\twoheadrightarrow}}

\ncm{\coequalizer}[1]{
\overset{#1}{
\parbox{40pt}{
\begin{picture}(40,8)
\put(2,4){\vector(1,0){32}} \put(37,4){\vector(1,0){0}}
\end{picture}
}}}

\ncm{\longcoequalizer}[1]{
\overset{#1}{
\parbox{60pt}{
\begin{picture}(60,8)
\put(2,4){\vector(1,0){52}} \put(57,4){\vector(1,0){0}}
\end{picture}
}}}

\ncm{\mono}[1]{\overset{#1}{\rightarrowtail}}

\ncm{\coequalizerfactorization}[9]{
\parbox[r]{115pt}{
\begin{picture}(115,70)(0,-5)
\put(0,48){$#1\longpair{#2}{#3}#4$} 
\end{picture}
}
\parbox[l]{80pt}{
\begin{picture}(80,70)(0,-5)
\put(2,51){\vector(1,0){70}} \put(75,51){\vector(1,0){0}}
\put(23,56){$\sst #5$}
\put(80,48){$#6$}
\put(2,47){\vector(2,-1){73}}
\put(48,30){$\sst #7$}
\dashline[+30]{3}(100,42)(100,12) \put(100,12){\vector(0,-1){0}}
\put(105,30){$\sst #8$}
\put(95,0){$#9$}
\end{picture}
}}

\begin{document}

\title{Skew monoidal monoids}
\author{K. Szlach\'anyi}
\date{}
\address{Wigner Research Centre for Physics, Budapest}
\email{szlachanyi.kornel@wigner.mta.hu}
\thanks{Supported by the Hungarian Scientific Research Fund, OTKA 108384}

\begin{abstract}
Skew monoidal categories are monoidal categories with non-invertible `coherence' morphisms. As shown in a previous paper
bialgebroids over a ring $R$ can be characterized as the closed skew monoidal structures on the category $\Mod R$ 
in which the unit object is $R_R$. 
This offers a new approach to bialgebroids and Hopf algebroids. Little is known about skew monoidal structures on general categories. In the present paper we study the one-object case: skew monoidal monoids (SMM). We show that they possess a dual pair of bialgebroids describing the symmetries of the (co)module categories of the SMM. These bialgebroids are submonoids of their own base and are rank 1 free over the base on the source side. We give various equivalent definitions of SMM, study the structure of their (co)module categories and discuss the possible closed and Hopf structures on a SMM. 

\end{abstract}
\maketitle

\section{Three definitions of skew monoidal monoids}

Skew monoidal monoids are skew monoidal categories, as defined in \cite{SMC}, in which the underlying category has only one object, so it is a monoid.

If $A$ denotes the monoid and 1 denotes its unit element then a skew-monoidal structure on $A$ amounts to have a monoid morphism $A\x A\to A$, $\bra a, b\ket\mapsto a\smp b$, the skew monoidal product,
\begin{align}
\label{smp 1}
(ab)\smp(cd)&=(a\smp c)(b\smp d)\\
\label{smp 2}
1\smp 1&=1
\end{align}
and elements $\gamma$, $\eta$, $\eps$ of $A$, the "coherence" morphisms,
satisfying the naturality conditions
\begin{align}
\label{nat gam}
\gamma (a\smp (b\smp c))&=((a\smp b)\smp c)\gamma\\
\label{nat eta}
\eta a&=(1\smp a)\eta\\
\label{nat eps}
a\eps&=\eps(a\smp 1)
\end{align}
and the skew monoidality axioms
\begin{align}
(\gamma\smp 1)\gamma(1\smp\gamma)&=\gamma^2\\
\gamma\eta&=\eta\smp 1\\
\eps\gamma&=1\smp \eps\\
(\eps\smp 1)\gamma(1\smp\eta)&=1\\
\eps\eta&=1\,.
\end{align}

If the skew monoidal product were associative, so $\gamma$ were the identity, and it had a unit then the Eckman-Hilton argument
would force $a\smp b=ab$ and $A$ would be commutative. Since $\gamma$ is not assumed even to be invertible, skew monoidal 
monoids have a good chance to be non-trivial.

\begin{lem} \label{lem: gam axioms}
A skew monoidal monoid $\A=\bra A,\smp,\gamma,\eta,\eps\ket$ is the same as the data consisting of
\begin{enumerate}
\item a monoid $A$, 
\item two monoid endomorphisms $T:A\to A$ and $Q:A\to A$ such that
\begin{equation} \label{smm-g 1}
T(a)Q(b)=Q(b)T(a)\qquad a,b\in A
\end{equation}
\item an element $\gamma\in A$ satisfying the intertwiner relations
\begin{align}
\label{smm-g 2}
\gamma T^2(a)&=T(a)\gamma\\
\label{smm-g 3}
\gamma TQ(a)&=QT(a)\gamma\\
\label{smm-g 4}
\gamma Q(a)&=Q^2(a)\gamma
\end{align}
for all $a\in A$
\item and elements $\eta,\eps\in A$ satisfying
\begin{align}
\label{smm-g 5}
\eta a&=T(a)\eta\\
\label{smm-g 6}
a\eps&=\eps Q(a)
\end{align}
for all $a\in A$.
\end{enumerate}
These data are required, furthermore, to obey the skew monoidality relations
\begin{align}
\label{smm-g 7}
Q(\gamma)\gamma T(\gamma)&=\gamma^2\\
\label{smm-g 8}
\gamma\eta&=Q(\eta)\\
\label{smm-g 9}
\eps\gamma&=T(\eps)\\
\label{smm-g 10}
Q(\eps)\gamma T(\eta)&=1\\
\label{smm-g 11}
\eps\eta&=1\,.
\end{align}
\end{lem}
\begin{proof}
Assume $\bra A,\smp,\gamma,\eta,\eps\ket$ is a SMM. Define $T:=1\smp\under$ and $Q:=\under\smp 1$. Then $T$ and $Q$
are monoid endomorphisms by (\ref{smp 1}), (\ref{smp 2}) and satisfies (\ref{smm-g 1}) by the interchange law, the latter being
a consequence of (\ref{smp 1}), too.
Replacing two of the $a,b,c$ by 1 in the expression $a\smp (b \smp c)$ we obtain $T^2(c)$, $TQ(b)$ and $Q(a)$, respectively.
In a similar way $(a\smp b)\smp c$ leads to $T(c)$, $QT(b)$ and $Q^2(a)$. Therefore the naturality condition (\ref{nat gam}) alone
implies (\ref{smm-g 2}), (\ref{smm-g 3}) and (\ref{smm-g 4}). Equations (\ref{smm-g 5}) and (\ref{smm-g 6}) are obvious transcriptions
of (\ref{nat eta}) and (\ref{nat eps}). The same can be said about the remaining 5 skew-monoidality relations.

Vice versa, assume that the data $\bra A, T,Q,\gamma,\eta,\eps\ket$ satis fies the 11 axioms (\ref{smm-g 1}-\ref{smm-g 11}).
Then a skew monoidal product on the one-object category $A$ can be defined by $a\smp b:=Q(a)T(b)$.  The verification of the skew
monoidal monoid axioms is now easy.
\end{proof}

Let $\Mon$ denote the 2-category of monoids as a monoidal sub-2-category of the cartesian monoidal 2-category $\Cat$.
So the objects of $\Mon$ are the monoids, the 1-cells are the monoid homomorphisms $f:A\to B$ and the 2-cells $b:f\to g:A\to B$ are the elements $b\in B$ satisfying the intertwiner property $bf(a)=g(a)b$ for all $a\in A$.

As in every skew monoidal category (see \cite[Lemma 2.6]{SMC}) we can introduce a monad $\bra T,\mu,\eta\ket$ and a comonad $\bra Q,\delta,\eps\ket$ on the object $A$ of $\Mon$ where multiplication $\mu$ and the comultiplication $\delta$ are defined by
\begin{align}
\label{def of mu}
\mu&:=Q(\eps)\gamma\\
\label{def of delta}
\delta&:=\gamma T(\eta)\,.
\end{align}
It has also been shown in that Lemma that $\gamma:TQ\to QT$ is a mixed distributive law. In our one-object category the defining
relations of the distributive law take the form
\begin{align}
Q(\mu)\gamma T(\gamma)&=\gamma\mu\\
Q(\gamma)\gamma T(\delta)&=\delta\gamma\\
\gamma\eta&=Q(\eta)\\
\eps\gamma&=T(\eps)\,.
\end{align}
The last two of these coincide with the triangle relations (\ref{smm-g 8}-\ref{smm-g 9}), the first two are simple consequences of the pentagon relation (\ref{smm-g 7}).

The SMM axioms can be reformulated using $\mu$ and $\delta$ instead of $\gamma$. Although the number of axioms gets
larger, it reveals SMM-s as something like a bimonoid. 
\begin{pro} \label{pro: SMM mu-del axioms}
A skew monoidal monoid is the same as the data consisting of
\begin{itemize}
\item a monoid $A$,
\item two monoid endomorphisms $T:A\to A$ and $Q:A\to A$
\item and elements $\mu$, $\eta$, $\delta$, $\eps$ of $A$
\end{itemize}
subject to the following axioms.

Intertwiner relations: For all $a\in A$
\begin{align}
\label{SMM1}
\mu\, T^2(a)&=T(a)\mu\\
\label{SMM2}
\eta \,a&=T(a)\eta\\
\delta\, Q(a)&=Q^2(a)\delta\\
a\,\eps&=\eps\, Q(a)\\
\label{SMM5}
\mu\,\delta\, TQ(a)&=QT(a)\mu\,\delta
\end{align}

Commutation relations: For all $a,b\in A$
\begin{align}
\label{SMM6}
T(a)Q(b)&=Q(b)T(a)\\
\label{SMM7}
\mu Q(a)&=Q(a)\mu\\
\label{SMM8}
\delta T(a)&=T(a)\delta
\end{align}

Bimonoid(-like) relations:
\begin{align}
\label{SMM9}
\mu T(\mu)&=\mu^2\\
\label{SMM10}
\mu\,\eta&=1\\
\label{SMM11}
\mu\, T(\eta)&=1\\
\label{SMM12}
Q(\delta)\delta&=\delta^2\\
\label{SMM13}
\eps\delta&=1\\
\label{SMM14}
Q(\eps)\delta&=1\\
\label{SMM15}
Q(\mu)\mu\,\delta T(\delta)&=\delta\,\mu\\
\label{SMM16}
\eps\,\mu&=\eps\, T(\eps)\\
\label{SMM17}
\delta\,\eta&=Q(\eta)\eta\\
\label{SMM18}
\eps\eta&=1
\end{align}
\end{pro}

\begin{proof}
Setting $\gamma:=\mu\delta$ the (\ref{SMM1}-\ref{SMM18}) axioms imply the axioms of Lemma \ref{lem: gam axioms}.
Assuming the latter axioms and defining $\mu$ and $\delta$ by (\ref{def of mu}) and (\ref{def of delta}) the axioms
(\ref{SMM1}-\ref{SMM18}) are easily obtained.
\end{proof}

Every SMM $A$ contains distinguished submonoids
\begin{align}
Q(A)\ \subset\ F\ \supset\ S(A)\ \subset\ G\ \supset T(A)
\end{align}
where
\begin{align}
F&=Q(A)\delta\\
G&=\mu T(A)
\end{align}
and $S$ is the anti endomorphism 
\begin{equation}
\label{S}
S:A\to A^\op,\quad S(a)=\mu T(Q(a)\eta)\equiv Q(\eps T(a))\delta\,.
\end{equation}
Furthermore,
\begin{enumerate}
\item the endomorphisms $Q$, $S$, $T$ are all injective, 
\item $\pi:A\to A$, $\pi(a):=\eps a\eta$ is a common left inverse in $\Set$ of the endofunctions $Q$, $S$ and $T$,
\item their images, $Q(A)$, $S(A)$ and $T(A)$, pairwise commute,
\item $S(A)=F\cap G$.
\end{enumerate}
For example, (iv) can be shown as follows. By the two equivalent formulas in (\ref{S}) it is clear that $S(A)\subset F\cap G$.
Vice versa, suppose $x\in F\cap G$. Then there exist $a,b\in A$ such that $x=Q(a)\delta=\mu T(b)$. Multiplying with $\eps$, resp. $\eta$ we obtain $a=\eps x$ and $b=x\eta$. Introducing $y:=a\eta$ we can write  $Q(y)\eta=Q(a)Q(\eta)\eta=Q(a)\delta\eta=x\eta=b$ and therefore $S(y)=\mu T(Q(y)\eta)=\mu T(b)=x$. This proves 
$F\cap G\subset  S(A)$.

The following useful relations will also be used in the sequel without explicit mention:
\begin{align}
\label{mu S}
\mu S(a)&=\mu TQ(a)\\
S(a)\mu&=\mu TS(a)\\
\label{S delta}
S(a)\delta&=QT(a)\delta\\
\delta S(a)&=QS(a)\delta\\
\label{S eta}
S(a)\eta&=Q(a)\eta\\
\label{eps S}
\eps S(a)&=\eps T(a)\,.
\end{align}

\section{The category of modules}

The category $\Mod\A$ of modules over a skew monoidal monoid $\A=\bra A,\smp,\gamma,\eta,\eps\ket$ is by definition \cite{SMC} the Eilenberg-Moore category $\A^T$ of the canonical monad $T$. Since $\A^T$ is just a category while the category 
of modules over a bialgebroid is always monoidal, the challenge is to find a monoidal structure on $\A^T$. 
This problem has been solved for skew monoidal categories \cite{Sz: Swansea}, \cite[Theorem 8.1]{LS:normalization} provided the
underlying category has reflexive coequalizers and the skew monoidal product preserves such coequalizers in the 2nd argument. As we shall see, in case of skew monoidal monoids these assumptions are not necessary.

Since $\A$ has only one object, the objects of $\A^T$ are just elements $x\in A$ satisfying two equations
\begin{align}
\label{T-alg1}
x\mu &=xT(x)\\
\label{T-alg2}
x\eta&=1_A\,.
\end{align}
The arrows $t\in\A^T(x,y)$ are the elements $t\in A$ such that
\begin{equation}
\label{T-alg-mor}
tx=y T(t)\,.
\end{equation}

For a skew monoidal category the monoidal product, called the horizontal tensor product and denoted by $\hot$, of two $T$-algebras
$\alpha: TM\to M$ and $\beta:TN\to N$ is defined as follows. First  we introduce the notation $\mu_{M,N}:=(\eps_M\smp N)\ci\gamma_{M,R,N}$. Then the underlying object of $M\hot N$ is given by the reflexive coequalizer
\[
M\smp TN\longpair{\mu_{M,N}}{M\smp \beta}M\smp N\coequalizer{}M\hot N
\]
and the $T$-action on $M\hot N$ is given by unique factorization in the diagram
\begin{alignat}{2}
&R\smp(N\smp TM)\longpair{R\smp\mu_{N,M}}{R\smp(N\smp \alpha)}
&R\smp(N\smp M) \longcoequalizer{R\smp\pi}
&R\smp(N\hot{M}) \notag\\ 
\label{def of psi}
&\quad\parbox[c]{5pt}{\begin{picture}(5,40)\put(2,40){\vector(0,-1){40}}\end{picture}}\sst(\beta\smp TM)\ci\gamma_{R,N,TM}
&\ \parbox[c]{5pt}{\begin{picture}(5,40)\put(2,40){\vector(0,-1){40}}\end{picture}}\sst(\beta\smp M)\ci\gamma_{R,N,M}\qquad
&\quad\parbox[c]{5pt}{\begin{picture}(5,40)\dashline{3}(2,40)(2,2)\put(2,2){\vector(0,-1){0}}\end{picture}}\sst\psi\\
&N\smp TM\quad\quad\longpair{\mu_{N,M}}{N\smp \alpha}
&N\smp M \longcoequalizer{\pi}
&N\hot{M} \notag
\end{alignat}
provided $T$ preserves reflexive coequalizers. In case of skew monoidal monoids all these diagrams simplify radically: all objects
are the same and all $\mu$-s are the same element (\ref{def of mu}) of $A$. Thus the horizontal tensor product of
$x,y\in\A^T$ is the element $x\hotobj y$ of $A$ making the diagram in $\A$ 
\[
\parbox{120pt}{
\begin{picture}(120,85)(0,-7)
\put(-2,0){$A$}
\put(11,0){$\pair{\mu}{T(y)}$}
\put(60,0){$A$} 
\put(74,0){$\coequalizer{y}$}
\put(120,0){$A$}

\put(-2,60){$A$}
\put(11,60){$\pair{T(\mu)}{T^2(y)}$}
\put(60,60){$A$} 
\put(74,60){$\coequalizer{T(y)}$}
\put(120,60){$A$}

\put(3,55){\vector(0,-1){43}} \put(8,31){$\sst Q(x)\gamma$}
\put(65,55){\vector(0,-1){43}} \put(70,31){$\sst Q(x)\gamma$}
\put(125,55){\vector(0,-1){43}} \put(130,31){$\sst x\hotobj y$}
\end{picture}}
\]
commutative. Now it is easy to check that
\begin{equation}
x\hotobj y=y Q(x)\delta\,.
\end{equation}
Note that the 2nd row is the canonical split coequalizer of the $T$-algebra $y$, hence it is preserved by $T$. This proves that the horizontal tensor product of any pair of objects exists in $\A^T$ for any skew monoidal monoid $\A$. 

It is now an easy exercise to determine the horizontal tensor product of arrows $s\in\A^T(x_1,y_1)$ and $t\in\A^T(x_2,y_2)$.
It is given by the formula
\begin{equation}
s\hot t =y_2Q(s)\eta t\quad\in\ \A^T(x_1\hotobj x_2,y_1\hotobj y_2)\,.
\end{equation}
Notice that we use different notations for horizontal tensor product of objects and of arrows. This is necessary since the same element of $A$ can be an object and an (non-identity) arrow in $\A^T$.

\begin{pro} \label{pro: forg fun of modules}
Let $\A=\bra A, T, Q,\gamma,\eta,\eps\ket$ be a skew monoidal monoid. Then its module category $\A^T$ equipped with the horizontal tensor product is a strict monoidal category. In this category $\mu$ is the underlying object of a comonoid
\[
\eps\larr{\eps} \mu\rarr{\delta} \mu\hotobj \mu\,.
\]
The associated representable functor is a faithful strong monoidal functor 
\[
\A^T(\mu,\under)\ :\ \A^T\ \to\ \,_A\Set_A\,,\qquad x\ \mapsto\ xT(A)
\]
to the category of $A$-$A$-bimodules (=$A$-bisets or $(A^\op\x A)$-sets).
\end{pro}
\begin{proof}
Strict associativity of $\hotobj$ and $\hot$ is a routine calculation. So is unitality; the unit object is $\eps$.
By (\ref{SMM9}-\ref{SMM10}) the $\mu$ is an object of $\A^T$. It has a special property that for any object $x\in\A^T$ the element
$x\in A$ is also an arrow $x:\mu\to x$ in $\A^T$. Indeed, compare (\ref{T-alg1}) with (\ref{T-alg-mor}).
It follows that the functor $\A^T(\mu,\under)$ is faithful: If $sx=tx$ for a pair $s,t$ of parallel arrows $x\to y$ then $s=t$ by
(\ref{T-alg2}). Now $\eps:\mu\to\eps$ is an arrow of $\A^T$ and so is $\delta:\mu\to\mu\hotobj \mu$ by (\ref{SMM15}).
The comonoid axioms for $\bra\mu,\delta,\eps\ket$ reduce to the comonad axioms (\ref{SMM12}-\ref{SMM14}) after noticing that
for any arrow $t\in\A^T$  we have $1_\mu\hot t=t$ and $t\hot 1_\mu=Q(t)$, as elements of $A$.

Before turning to the monoidal structure of the functor $\A^T(\mu,\under)$ we need an explicit computation of the intertwiner
set $\A^T(\mu, x)$. If $s:\mu\to x$ is an intertiner then $s=s\mu T(\eta)=xT(s\eta)\in xT(A)$. Vice versa, if $a\in A$ then
$xT(a)\mu=x\mu T^2(a)=xT(xT(a))$, so $xT(a)$ is an intertwiner $\mu\to x$. This proves the equality $\A^T(\mu,x)=xT(A)$,
as subsets of $A$. In particular, $\A^T(\mu,\mu)=\mu T(A)$ is a monoid containing both $T(A)$ and $S(A)$ as submonoids.
This allows us to define the $A$-$A$-bimodule structure on $\A^T(\mu,x)$ by right multiplication
\[
a_1\cdot t\cdot a_2\ :=\ tS(a_1)T(a_2),\qquad t\in\A^T(\mu,x),\ a_1,a_2\in A\,.
\]
The comonoid structure of $\mu$ induces the following monoidal structure for the functor $\A^T(\mu,\under)$. 
\begin{align}
\notag\A^T(\mu,x)\oA \A^T(\mu,y)&\to\A^T(\mu,x\hotobj y)\\
\label{Phi_2}
s\oA t&\mapsto (s\hot t)\ci\delta\ =\ tQ(s)\delta
\end{align}
and the arrow $A\to\A^T(\mu,\eps)$ in $_A\Set_A$ is given by $1_A\mapsto \eps$ which is well-defined since
$a\cdot\eps=\eps S(a)=\eps T(a)=\eps\cdot a$ by (\ref{eps S}). 
This monoidal structure is strong: The inverse of (\ref{Phi_2}) is
\[
r\in\A^T(\mu, x\hotobj y)\ \mapsto\ x\oA y T(r\eta)\ \in\ \A^T(\mu,x)\oA \A^T(\mu,y)
\]
and the inverse of the identity constraint is 
\[
r\in\A^T(\mu,\eps)\ \mapsto r\eta \in A
\]
since $\A^T(\mu,\eps)=\eps T(A)$.
\end{proof}

Obviously, the next task is to factorize the forgetful functor of the above Proposition through the forgetful functor of $G$-modules,
where $G=\A^T(\mu,\mu)$, and show that $G$ is a right bialgebroid over $A$. The conspicuous property of this bialgebroid is that it is a submonoid of its own base monoid.

\begin{thm} \label{thm: G}
The forgetful functor of Proposition \ref{pro: forg fun of modules} factorizes through the forgetful functor $\Set_G\to \,_A\Set_A$
of a right $A$-bialgebroid $G$ via a fully faithful $\A^T\to \Set_G$ where the bialgebroid $G$ is defined by
\begin{align*}
\text{underlying monoid}&:=\A^T(\mu,\mu)=\mu T(A)\\
\text{source } s^G:A\to G,\quad s^G(a)&:=T(a)\\
\text{target } t^G:A^\op\to G,\quad t^G(a)&:=S(a)\\
\text{comultiplication } \Delta^G:G\to G\oA G,\quad \Delta^G(g)&:=\mu\oA \mu T(\delta g \eta)\\
\text{counit }\epsilon^G:G\to A,\quad \epsilon^G(g)&:=\pi(g)=\eps g\eta\,.
\end{align*}
\end{thm}
\begin{proof}
The verification of the right bialgebroid axioms is a short exercise. The functor $\A^T\to \Set_G$ is faithful since the functor $\A^T\to \,_A\Set_A$ was already faithful. In order to prove fullness let $f:xT(A)\to yT(A)$ be a $G$-module map, $f(sg)=f(s)g$.
Then $f(x)\mu=f(x)T(x)$ and, writing $f(x)=yT(t)$ with a unique $t\in A$, we obtain $yT(t)\mu=yT(tx)$ hence $yT(t)=tx$.
This proves that $t\in\A^T(x,y)$ and $f(s)=f(xT(s\eta))=f(x)T(s\eta)=yT(ts\eta)=ts$ therefore the functor is full.
\end{proof}
As usual for right bialgebroids we consider $G$ as an $A$-bimodule via $a_1\cdot g\cdot a_2=gS(a_1)T(a_2)$. 
Notice that $\mu$ is a free generator for $G$ as right $A$-module, the $G_A$ is free of rank 1. This explains the simple formula for the coproduct since $G\oA G\cong G$ at least as right $A$-modules. The left $A$-module structure is by far not so trivial. 

\begin{defi}
A right module $M$ over a right $A$-bialgebroid $G$ (or a left module over a left bialgebroid) is called source-regular if the $A$-module obtained from $M$ by restriction along the source map $s^G:A\to G$ is isomorphic to the right regular $A$-module $A_A$ (resp. the left regular $_AA$).
\end{defi}

\begin{cor} \label{cor: Set_G^reg}
Let $\A$ be a SMM. Then there is a source regular right bialgebroid $G$ such that $\Mod \A\equiv\A^T$ is monoidally equivalent 
to the category $\Set_G^{\reg}$ of source regular right $G$-modules.
\end{cor} 
\begin{proof}
Clearly, the image of the forgetful functor $\A^T(\mu,\under)$ consist only of the source-regular modules $xT(A)$ for
$x\in\ob\A^T$. 

Let $M$ be a source-regular right $G$-module and let $m_0\in M$ be a free generator of $M_A$, so $A\to M$, $a\mapsto 
m_0\cdot T(a)$ is a bijection. Then $m_0\cdot\mu=m_0\cdot T(x)$ for a unique $x\in A$. From the equations
\[
m_0\cdot T(x\mu)=m_0\cdot \mu T(\mu)=m_0\cdot \mu^2=m_0\cdot T(x)\mu=m_0\cdot \mu T^2(x)=m_0\cdot T(xT(x))
\]
it follows that $x\mu=xT(x)$. Similarly, $m_0\cdot T(x\eta)=m_0\cdot\mu T(\eta)=m_0$ implies $x\eta=1$. This proves that
$x$ is a $T$-algebra such that $M\cong xT(A)$ as $G$-modules.
\end{proof}

\section{Comodules and the dual bialgebroid} \label{sec: comod}

We shall be very brief about comodules since all what we have to say can be obtained from the structure of modules by dualizing,
i.e., $\Comod \A=\A_Q$ is nothing but $\Mod \A^{\op,\rev}$.

The objects of $\A_Q$ are elements $u\in A$ satisfying $Q(u) u=\delta u$ and $\eps u=1$. The arrows $s\in\A_Q(u,v)$ are
elements $s\in A$ satisfying $Q(s)u=vs$.

\begin{pro}
The Eilenberg-Moore category $\A_Q$ of the comonad $Q$ is a strict monoidal category w.r.t. the vertical tensor product defined for objects by $u\votobj v:= \mu T(v)u$ and for arrows $s:u_1\to v_1$, $t:u_2\to v_2$ by $s\vot t:=t\eps T(s)u_2$. 
The element $\delta\in A$ is the underlying object of the monoid $\bra \delta,\mu,\eta\ket$ in $\A_Q$ and $\delta$ is a cogenerator in $\A_Q$.
The endomorphism monoid of $\delta$ is the underlying monoid of a source-regular left $A$-bialgebroid $F$ defined by
\begin{align*}
\text{underlying monoid}\ :\ \A_Q(\delta,\delta)&=Q(A)\delta\\
\text{source } s_F:A\to F,\quad s_F(a)&=Q(a)\\
\text{target } t_F:A^\op\to F,\quad t_F(a)&=S(a)\\
\text{comultiplication } \Delta_F:\,_AF_A\to \,_A(F\oA F)_A,\quad \Delta_F(f)&=Q(\eps f\mu)\delta\oA \delta\\
\text{counit }\epsilon_F:\,_AF_A\to \,_AA_A,\quad \epsilon_F(f)&=\pi(f)\equiv\eps f\eta
\end{align*}
where in the last two lines the $A$-bimodule structure of $F$ is defined by $a_1\cdot f\cdot a_2:= Q(a_1)S(a_2)f$.
\end{pro}

\begin{thm}
The forgetful functor $\A_Q(\under,\delta) :\A_Q^\op\to\,_A\Set_A$ is strong monoidal and factors through a monoidal equivalence
$\A_Q^\op\cong \,_F^\reg\Set$ of $\A_Q^\op$ with the category of source-regular left $F$-modules.
\end{thm}

In order to characterize $\A_Q$ in terms of $G$ we need to perform a duality transformation also on the level of bialgebroids.
\begin{lem}
The left dual $G^*=\Hom(G_A,A_A)$ of the right $A$-bialgebroid $G$ can be identified with the left $A$-bialgebroid $F$ by the
non-degenerate pairing 
\[
\bra\under,\under\ket\ :\ F\x G\ \to\ A,\qquad \bra f,g\ket:=\eps fg\eta
\]
satisfying the following properties
\begin{align}
\bra f,a_1\cdot g\cdot a_2\ket&=\bra fQ(a_1),g\ket a_2\\
\bra a_1\cdot f\cdot a_2,g\ket&=a_1\bra f,T(a_2)g\ket\\
\bra f,g g'\ket&=\bra f\oneB\cdot\bra f\twoB,g\ket,g'\ket\\
\label{F-mul}
\bra f f',g\ket&=\bra f,\bra f',g\oneT\ket\cdot g\twoT\ket\\
\bra f, 1_G\ket&=\pi(f)\\
\label{F-uni}
\bra 1_F,g\ket&=\pi(g)\,.
\end{align}
\end{lem}
\begin{cor}
The category $\A_Q$ of comodules of the SMM $\A$ is monoidally equivalent to the category $\Set_\reg ^G$ of source-regular 
right $G$-comodules.
\end{cor}
\begin{proof}
By dualizing Corollary \ref{cor: Set_G^reg} every source-regular left $F$-module $X\in\,_F^\reg \Set$ is isomorphic to $Q(A)u$
for some $u\in\ob\A_Q$. The right dual $^*X=\Hom(_AX,\,_AA)$ is therefore isomorphic to 
\[
^*[Q(A)u]=\Hom(Q(A)u,\,_AA)=\{\varphi:Q(A)u\to A\,|\, \exists a_0\in A,\ \varphi(x)=\eps x a_0\,\}
\]
and left $F$-module structures on $X$ and right $G$-comodule structures on $^*X$ are in bijection via
\[
\bra f\cdot x,\varphi\ket=\bra f,\bra x,\varphi\nulT\ket\cdot\varphi\oneT\ket,\qquad f\in F,\ x\in X,\varphi\in\,^*X\,.
\]
For source-regular $F$-modules $X$ and $Y$ and for a left $A$-module map $h:X\to Y$ its transpose $^*h:\,^*Y\to\,^*X$ is a right $A$-module map and
\begin{align*}
\bra f\cdot h(x),\varphi\ket&=\bra f,\bra x,\,^*h(\varphi\nulT)\ket\cdot\varphi\oneT\ket\\
\bra h(f\cdot x),\varphi\ket&=\bra f,\bra x,\,^*h(\varphi)\nulT\ket\cdot\,^*h(\varphi)\oneT\ket
\end{align*}
Therefore $h$ is left $F$-linear precisely when $^*h$ is right $G$-colinear.
\end{proof}

Together with $F$ and $G$  the monoid $A$ contains also an image of the smash product $G\mash F$.
Indeed, the axiom (\ref{SMM15}) implies that for all $g\in G$, $f\in F$
\begin{align*}
fg&=Q(\eps f)\delta\mu T(g\eta)=Q(\eps f \mu)\mu\delta T(\delta g\eta)=\mu T(\delta g \eta) Q(\eps f\mu)\delta=\\
&=g\twoT S(\bra f\twoB,g\oneT\ket) f\oneB =(f\twoB\rightharpoonup g) f\oneB
\end{align*}
where $f\rightharpoonup g$ is the action that makes $G$ a left module algebra over the cooposite (left bialgebroid) of $F$.
The invariant submonoid is
\[
G^F=\{g\in G\,|\,f\rightharpoonup g=\pi(f)\cdot g\}\ =\ T(A)\,.
\]
The smash product $G\mash F=G\am{S(A)}F$ is the tensor product over the common submonoid $S(A)=F\cap G$ with
multiplication rule
\[
(g\mash f)(g'\mash f')=g(f\twoB\rightharpoonup g')\mash f\oneB f' =g{g'}\twoT S(\bra f\twoB,{g'}\oneT\ket)\mash f\oneB f'\,.
\]

There is a remarkable connection between grouplike elements of $G$ and the objects of $\A_Q$. Recall that in a bialgebroid $G$ an element $g\in G$ is called grouplike if $\Delta^G(g)=g\oA g$ and $\epsilon^G(g)=1$. 

\begin{lem}
The bijection $A\iso G$, $a\mapsto \mu T(a)$, restricts to a bijection $A_Q\iso \Gr(G)$ between the set of objects $u\in\A_Q$
and the set of grouplike elements of the bialgebroid $G$. This bijection lifts to an isomorphism of monoids:
$A_Q\iso \Gr(G)^\op$.
Dually, the map $x\mapsto Q(x)\delta$ defines an isomorphism of monoids $A^T\iso\Gr(F)^\op$ from the monoid of objects of $\A^T$ to the opposite of the monoid of grouplikes in $F$.
\end{lem}

We do not know other objects than the tensor powers of $\mu$ in $\A^T$ and those of $\delta$ in $\A_Q$.
The known part of $\A^T$, for example, contains the simplicial object
$$
\begin{picture}(300,100)
\put(0,47){$\eps$}
\put(67,50){\vector(-1,0){50}} \put(41,53){$\sst\eps$}

\put(75,47){$\mu$}
\put(143,63){\vector(-1,0){50}} \put(112,66){$\sst Q(\eps)$}
\put(93,50){\vector(1,0){50}} \put(115,53){$\sst\delta$}
\put(143,37){\vector(-1,0){50}} \put(115,40){$\sst\eps$}

\put(153,47){$\mu\hotobj\mu$}
\put(230,76){\vector(-1,0){50}} \put(197,79){$\sst Q^2(\eps)$}
\put(180,63){\vector(1,0){50}} \put(197,66){$\sst Q(\delta)$}
\put(230,50){\vector(-1,0){50}} \put(197,53){$\sst Q(\eps)$}
\put(180,37){\vector(1,0){50}} \put(202,40){$\sst\delta$}
\put(230,24){\vector(-1,0){50}} \put(202,27){$\sst\eps$}

\put(242,47){$\mu\hotobj\mu\hotobj\mu$}

\put(292,50){$\dots$}
\end{picture}
$$
associated to the comonad $\under\hot 1_\mu$ but it contains also endoarrows $g\in G=\A^T(\mu,\mu)$ at all occurence of the object $\mu$. E.g., there are arrows $\mu:\mu\to\mu$, $\mu\hot 1_\mu:\mu\hotobj\mu\to\mu\hotobj\mu$, \dots which do not
belong to the above simplicial object. There is also another simplicial object $\Delta^\op\to\A^T$ associated to the comonad
$1_\mu\hot\under$. Dually, the comodule category $\A_Q$ contains cosimplicial objects and endoarrows $f\in F$ of the cogenerator object $\delta$. The conjecture is that the category generated by these objects and arrows exhausts all of $\A^T$ and $\A_Q$, respectively, when $\A$ is the free SMM.

\section{Skew monoidal monoids are source regular bialgebroids}

In the previous sections we have seen that every SMM $\A$ determines a dual pair of source regular bialgebroids $F$ and $G$.
In this section we show that every source regular bialgebroid determines a skew monoidal structure on its base monoid
and these two constructions are inverses of each other, up to isomorphisms.

First we need a precise notion of isomorphism of SMM-s.
Consider a skew monoidal functor $\bra A, T,Q,\gamma,\eta,\eps\ket\rarr{\Phi}\bra A', T',Q',\gamma',\eta',\eps'\ket$. Such a functor consists of a monoid morphism $\phi:A\to A'$ and elements $\phi_2,\phi_0\in A'$ 
satisfying the intertwiner relations (expressing naturality of $\phi_2$)
\begin{equation} \label{SMM iso1}
\phi_2 Q'\phi(a)\ =\ \phi Q(a)\phi_2\quad\text{and}\quad \phi_2 T'\phi(a)\ =\ \phi T(a)\phi_2
\end{equation}
and the identities (being the 3 skew monoidal functor axioms)
\begin{align}
\label{SMM iso2}
\phi(\gamma)\phi_2 T'(\phi_2)&=\phi_2 Q'(\phi_2)\gamma'\\
\label{SMM iso3}
\phi_2 Q'(\phi_0)\eta'&=\phi(\eta)\\
\label{SMM iso4}
\phi(\eps)\phi_2 T'(\phi_0)&=\eps'
\end{align}

\begin{defi} \label{def: SMM iso}
An isomorphism $\Phi:\A\iso\A'$ of skew monoidal monoids is a skew monoidal functor $\bra\phi,\phi_2,\phi_0\ket:\A\to\A'$ 
such that $\phi:A\to A'$ is an isomorphism of monoids and the elements $\phi_2$ and $\phi_0$ of $A'$ are invertible.
\end{defi}

In particular, an isomorphism of SMM-s is both a skew monoidal and a skew opmonoidal functor. Therefore it determines
not only a monad morphism
\begin{align}
\sigma&:=\phi_2 Q'(\phi_0)\ :\ T'\phi\to \phi T\\
\sigma\mu'&=\phi(\mu)\sigma T'(\sigma)\\
\sigma\eta'&=\phi(\eta)
\end{align}
but a comonad morphism
\begin{align}
\tau&:=T'(\phi_0^{-1})\phi_2^{-1}\ :\ \phi Q\to Q'\phi\\
\delta'\tau&=Q'(\tau)\tau\phi(\delta)\\
\eps'\tau&=\phi(\eps)
\end{align}
as well. 
\begin{pro} \label{pro: arrow map}
Let $\A$ and $\A'$ be skew monoidal monoids and let $G$, resp. $G'$, denote the right $A$-bialgebroid associated to $\A$ 
by Theorem \ref{thm: G} and the right $A'$-bialgebroid associated to $\A'$. 
If $\Phi:\A\to\A'$ is an isomorphism of skew monoidal monoids then the pair $\bra \varphi,\varphi_0\ket$, where
\begin{align*}
\varphi:G&\to G'\,\qquad g\mapsto \tau\phi(g)\tau^{-1}\\ 
\varphi_0:A&\to A'\,,\qquad a\mapsto \phi_0^{-1}\phi(a)\phi_0
\end{align*}
is an isomorphism of bialgebroids from $G$ to $G'$, i.e., the following identities hold for all $a\in A$ and $g\in G$:
\begin{align}
\label{bgd-mor 1}
\varphi (T(a))&=T'(\varphi_0(a))\\
\label{bgd-mor 2}
\varphi (S(a))&=S'(\varphi_0(a))\\
\label{bgd-mor 3}
\epsilon^{G'}(\varphi(g))&=\varphi_0(\epsilon^G(g))\\
\label{bgd-mor 4}
\varphi(g)\oneT\am{A'}\varphi(g)\twoT&=\varphi(g\oneT)\am{A'}\varphi(g\twoT)\,.
\end{align}
\end{pro}
\begin{proof}
Notice that for $g\in G$ the expression $\sigma^{-1}\phi(g)\sigma$ is a well-defined composition of arrows $\mu'\to \phi(\mu)\sigma\to\phi(\mu)\sigma\to\mu'$ in ${\A'}^{T'}$, therefore defines an element in $G'$. But $\varphi$ was defined
using $\tau$ instead of $\sigma^{-1}$. Fortunately, the difference is only an inner automorphism of $G'$, induced by $T'(\phi_0)$, 
so $\varphi:G\to G'$ is well-defined.

The verification of the relations (\ref{bgd-mor 1} - \ref{bgd-mor 4}) requires too much place to account for in full detail.
But it is straightforward using the (co)monad morphism relations and, occasionaly, the fact that since $\phi$ is an isomorphism, 
expressions like $\phi_2 Q'(\phi_0)\phi_2^{-1}$ commute with $\phi(g)$ for $g\in G$. 
\end{proof}

\begin{thm}
Let $A$ be a monoid. 
Then there is a bijection between isomorphism classes of
skew monoidal monoids $\bra A, T,Q,\mu,\eta,\delta,\eps\ket$ with underlying monoid $A$
and source regular right bialgebroids $\bra G,A,s^G,t^G,\cop^G,\epsilon^G\ket$ over $A$.
\end{thm}
\begin{proof}
The construction of $G$ from a SMM $\A$ given in Theorem \ref{thm: G} serves as the object map of a functor from the category of SMM-s, with arrows  being the isomorphisms of Definition \ref{def: SMM iso}, into the category of source regular bialgebroids. The arrow map is provided by Proposition \ref{pro: arrow map}. 

Next we construct a map from source regular bialgebroids to SMM-s. Source regularity is precisely the statement that the source map $s^G:A\to G$, as a 1-cell in $\Mon$, has a right adjoint. We choose a right adjoint $s_G:G\to A$ with unit $\eta:\id_A\to s_G s^G$ and counit $\mu:s^Gs_G\to\id_G$. 
Then both $s_G$ and $s^G$ are injective and $a\mapsto \mu s^G(a)$ is a bijection $A\to G$ with inverse $g\mapsto s_G(g)\eta$.
From the adjunction relations it follows that $T:=s_Gs^G$ is a monad on $A$ with multiplication $s_G(\mu)$ and
unit $\eta$. This proves 5 of the 18 SMM axioms of Proposition \ref{pro: SMM mu-del axioms}, namely (\ref{SMM1}), (\ref{SMM2}),
(\ref{SMM9}),  (\ref{SMM10}) and (\ref{SMM11}).

Let $F:=\Hom(G_A,A_A)$ with $\bra f,g\ket$ denoting evaluation of $f\in F$ on $g\in G$.
Then $F$ is a left $A$-module via $\bra a\cdot f,g\ket=a\bra f, g\ket$ and every $f\in F$ is uniquely determined by its value on $\mu$ since $\bra f,\mu s^G(a)\ket=\bra f,\mu\ket a$. This defines a bijection $J:F\to A$, $J(f):=\bra f,\mu\ket$. Using the comonoid structure of $G$ we make $F$ into a monoid by (\ref{F-mul}) and (\ref{F-uni}).
Then defining
\begin{equation}\label{eps-del-sF}
\eps:=J(1_F),\quad\delta:= J^{-1}(1_A),\quad s_F(a):=J^{-1}(a\eps),\quad s^F(f):=J(\delta f)
\end{equation}
we obtain that  $J(s_F(a))=a\eps=J(a\cdot 1_F)$, hence $s_F(a)=a\cdot 1_F$ and
\begin{align}
\bra fs_F(a),g\ket&=\bra f,\bra s_F(a),g\oneT\ket\cdot g\twoT\ket=\bra f,a\cdot g\ket=\notag\\
\label{13}&=\bra f,gt^G(a)\ket\,.
\end{align}
We deduce that
$s_F:A\to F$ and $s^F:F\to A$ are monoid morphisms and $s^F$ is left adjoint to the source map $s_F$ with
unit $\delta$ and counit $\eps$. It follows that $Q:=s^F s_F$ is a comonad on $A$ with comultiplication $s^F(\delta)$ and counit $\eps$. This proves 5 more SMM axioms.

The canonical pairing can now be written as 
\begin{equation}\label{26}
\bra f,g\ket\ =\ \eps s^F(f) s_G(g)\eta\,.
\end{equation}
The target map of the dual (would-be-)bialgebroid $F$ can be introduced by 
\begin{equation}\label{27}
\bra t_F(a),g\ket\ :=\ \bra 1_F,s^G(a)g\ket
\end{equation}
and proving
\begin{align} 
\notag\bra ft_F(a),g\ket&=\bra f,\bra t_F(a),g\oneT\ket\cdot g\twoT\ket=\bra f,\eps^G(s^G(a)g\oneT)\cdot g\twoT\ket\\
\notag&=\bra f,\eps^G(g\oneT)\cdot t^G(a)g\twoT\ket=\\
\label{29}&=\bra f,t^G(a)g\ket
\end{align}
it follows that it is a monoid morphism $t_F:A^\op\to F$ and $t_F(a)$ commutes with $s_F(b)$ for all $a,b\in A$. What used to be the antiendomorphism $S$ in previous Sections can now be written as
\begin{equation}\label{30}
s^Ft_F(a)=\bra\delta t_F(a),\mu\ket=\bra \delta,t^G(a)\mu\ket=s_Gt^G(a)\,.
\end{equation}
There is one more property of the pairing that we need:
\begin{align}\notag
\bra t_F(a)f,g\ket&=\bra t_F(a),\bra f,g\oneT\ket\cdot g\twoT\ket\eqby{27}\bra 1_F,\bra f,g\oneT\ket\cdot s^G(a)g\twoT\ket=\\
\notag&=\bra 1_F,\bra f,(s^G(a)g)\oneT\ket\cdot (s^G(a)g)\twoT\ket=\\
\label{31}&=\bra f,s^G(a),g\ket\,.
\end{align}
The basic commutation relations of a SMM structure can now be proven as follows. Since
\[
s^Fs_F(a)=\bra\delta s_F(a),\mu\ket\eqby{13}\bra\delta,\mu t^G(a)\ket\eqby{26}s_G(\mu)s_Gt^G(a)\eta\,,
\]
we can conclude that 
\begin{align}
\notag s^Fs_F(a)s_G(g)&=s_G(\mu)s_Gt^G(a)s_Gs^Gs_G(g)\eta=s_G(\mu)s_Gs^Gs_G(g)s_Gt^G(a)\eta=\\
\label{33}&=s_G(g)s^Fs_F(a)\,.
\end{align}
Similarly, from the equality
\[
s_Gs^G(a)=\bra \delta,s^G(a)\mu\ket\eqby{31}\bra t_F(a)\delta,\mu\ket=\eps s^Ft_F(a)s^F(\delta)
\]
we infer
\begin{align}
\notag s^F(f)s_Gs^G(a)&=\eps s^Fs_Fs^F(f) s^Ft_F(a)s^F(\delta)=\eps s^Ft_F(a)s^Fs_Fs^F(f) s^F(\delta)=\\
\label{35}&=s_Gs^G(a) s^F(f)\,.
\end{align}
Equations (\ref{33}) and (\ref{35}) prove, redundantly, the 3 SMM axioms (\ref{SMM6}), (\ref{SMM7}) and (\ref{SMM8}).

It remains to show the intertwiner relation (\ref{SMM5}) for $\gamma=s_G(\mu)s^F(\delta)$ and the last 4 bimonoid axioms.
For that first one proves the relations analogous to (\ref{mu S}), (\ref{S delta}), (\ref{S eta}) and (\ref{eps S}), namely
\begin{align}
\label{36}
s^Ft_F(a)\eta&=s^Fs_F(a)\eta\\
\label{37}
\mu t^G(a)&=\mu s^Gs^Fs_F(a)\\
\label{38}
\eps s_Gt^G(a)&=\eps s_Gs^G(a)\\
\label{39}
t_F(a)\delta&=s_Fs_Gs^G(a)\delta\,.
\end{align}
Then 
\begin{gather*}
s_G(\mu)s^F(\delta)\, s_Gs^Gs^Fs_F(a)\eqby{35}s_G(\mu)s_Gs^Gs^Fs_F(a)s^F(\delta)\eqby{37}\\
=s_G(\mu)s_Gt^G(a)s^F(\delta)\eqby{30}s_G(\mu)s^Ft_F(a)s^F(\delta)\eqby{39}s_G(\mu)s^Fs_Fs_Gs^G(a)s^F(\delta)\eqby{33}\\
=s^Fs_Fs_Gs^G(a)\, s_G(\mu)s^F(\delta)
\end{gather*}
which is precisely the intertwiner relation (\ref{SMM5}). Finally, the remaining bimonoid relations can be shown by the calculations
\begin{align*}
\eps\eta&\eqby{26}\bra 1_F,1_G\ket=\eps^G(1_G)=1_A\\
\eps s_G(\mu)&\eqby{26}\bra 1_F,\mu^2\ket=\epsilon^G(\mu^2)=\epsilon^G(s^G(\epsilon^G(\mu))\mu)\eqby{26}
\bra 1_F,s^G(\eps)\mu\ket=\\
&\eqby{31} \bra t_F(\eps),\mu\ket\eqby{26}\eps s^Ft_F(\eps)\eqby{30}\eps s_Gt^G(\eps)\eqby{38}\eps s_Gs^G(\eps)\\
s^F(\delta)\eta&\eqby{26}\bra\delta^2,1_G\ket=\bra\delta,\bra\delta,1_G\ket\cdot 1_G\ket\eqby{26}\bra\delta,t^G(\eta)\ket=\\
&\eqby{26}s_Gt^G(\eta)\eta\eqby{30}s^Ft_F(\eta)\eta\eqby{36}s^Fs_F(\eta)\eta\\
s^F(\delta)s_G(\mu)&\eqby{26}\bra\delta^2,\mu^2\ket=\bra\delta,\bra\delta,\mu\oneT\mu^{(1')}\ket\cdot \mu\twoT\mu^{(2')}\ket=\\
&=\bra\delta,\bra\delta,\mu^2\ket\cdot(\mu s^Gs^F(\delta))^2\ket=\bra\delta,s_G(\mu)\cdot \mu^2 s^Gs_Gs^Gs^F(\delta)
s^Gs^F(\delta)\ket=\\
&=\bra\delta,s_G(\mu)\cdot\mu^2\ket s_Gs^Gs^F(\delta)s^F(\delta)=\bra\delta,\mu^2 t^Gs_G(\mu)\ket s_Gs^Gs^F(\delta)s^F(\delta)=\\
&\eqby{37}\bra\delta,\mu^2 S^Gs^Fs_Fs_G(\mu)\ket s_Gs^Gs^F(\delta)s^F(\delta)=\\
&=s_G(\mu)s^Fs_Fs_G(\mu) s_Gs^Gs^F(\delta)s^F(\delta)
\end{align*}
where for the last relation, in passing from the first line to the second, we used the coproduct formula $\cop^G(\mu)=\mu\oA
\mu T(\delta)$ which, in the present construction, can be easily proven by pairing it with arbitrary $f$, $f'\in F$. 
This finishes the construction of a skew monoidal monoid
\[
\A=\bra A,s_Gs^G,s^Fs_F,s_G(\mu),\eta,s^F(\delta),\eps\ket
\]
from the data of a source regular right $A$-bialgebroid $G$ and from a choice of adjunction data $s_G$, $\mu$, $\eta$ for
the left adjoint $s^G$. 

If $G$ is obtained from a SMM structure on $A$, let us say $\A'$, then $s^G$ has another right adjoint, the inclusion 
$s'_G: G\into A$ with counit $\mu'$ and unit $\eta'$. Then there is an invertible 2-cell $z:s_G\to s'_G$ in $\Mon$  
such that $\mu' s^G(z)=\mu$ and $z\eta=\eta'$. Then (\ref{eps-del-sF}) yields $\eps'=\eps z^{-1}$, $\delta'=s_F(z)\delta$,
$s'_F(a)=s_F(a)$ and ${s'}^F(f)=z s^F(f) z^{-1}$.
It follows that the monoid automorphism $\phi(a)=z a z^{-1}$ together with
\[
\phi_2=Q'(z^{-1})T'(z^{-1})\quad\text{and}\quad \phi_0=z
\]
is an isomorphism of skew monoidal monoids from $\A$ to $\A'$, that is to say the triple $\bra\phi,\phi_2,\phi_0\ket$ satisfies equations (\ref{SMM iso1}), (\ref{SMM iso2}), (\ref{SMM iso3}) and (\ref{SMM iso4}).

Going in the opposite direction suppose that we start from a source regular bialgebroid $G$, perform the above construction of a
SMM $\A$ and then apply Theorem \ref{thm: G} to construct a bialgebroid $G'$. This bialgebroid $G'$ is nothing but the submonoid of $A$ generated by $s_G(\mu)$ and $s_Gs^G(A)$. But this is nothing but the image of the map $s_G$. Therefore
restricting $s_G$ to its image (together with the identity $A\to A$) is a bialgebroid isomorphism $G\iso G'$.
\end{proof}

\section{Closed and Hopf skew monoidal monoids}

Let $\M$ be the category of all (small) right $A$-modules. Then the $A$-bialgebroid $G$ of a SMM $\A$ determines a closed skew 
monoidal structure on $\M$ with skew monoidal product
\begin{equation}
\label{smpc}
M\smpc N:=M\oS(N\oT G)
\end{equation}
where $\oS$ refers to tensoring over $A$ w.r.t the left $A$-action $a\cdot g= gS(a)$ on $G$ and $\oT$ w.r.t. the other left
$A$-action $a\cdot g=T(a)g$. That is to say, $M\smpc N$ is the right $A$-set the elements of which are equivalence classes
$[m,n,g]$ of elements $\bra m,n,g\ket\in M\x N\x G$ w.r.t. the equivalence relation generated by 
\[
\bra m\cdot b,n\cdot a,g\ket\ \sim\ \bra m, n,T(a)gS(b)\ket\,\quad m\in M,\ n\in N,\ g\in G,\ a,b\in A\,.
\]
The $A$-action is given by $[m,n,g]\cdot a:=[m,n,gT(a)]$.
The skew unit object $R$ is the right regular $A$-module and the coherence morphisms are
\begin{align*}
\gamma_{L,M,N}\ :\ L\smpc(M\smpc N)&\to (L\smpc M)\smpc N\\
[l,[m,n,g],h]&\mapsto[[l,m,h\oneT],n,gh\twoT]\\
\eta_M\ :\ M&\to R\smpc M\\
m&\mapsto[1_A,m,1_G]\\
\eps_M\ :\ M\smpc R&\to M\\
[m,a,g]&\mapsto m\cdot \pi(T(a)g)\,.
\end{align*}
The skew monoidal product $M\smpc N$ being a colimit, both endofunctors $M\smpc\under$ and $\under \smpc N$ on $\M=\Set_A$ have right adjoints therefore $\smpc$ is a closed skew monoidal structure.

Let $\M_1\subset\M$ be the full subcategory of rank 1 free $A$-modules. For each object $M\in\M_1$ choose a free generator $\xi_M\in M$,
so that 
\[
\nu_M:R\to M,\qquad a\mapsto \xi_M\cdot a
\]
is an isomorphism in $\M_1$. For two objects $M,N\in\M_1$ define
\[
\nu_{M,N}:M\smpc N\to G,\quad [m,n,g]\mapsto T(\nu_N^{-1}(n)) g S(\nu_M^{-1}(m))
\]
which is an isomorphism in $\M$ with inverse $\nu_{M,N}^{-1}(g)=[\xi_M,\xi_N,g]$. Since $G\in\M_1$ with $\xi_G=\mu$, the category $\M_1$
is a skew monoidal subcategory of $\M$. 

\begin{lem}\label{lem: M_1}
The functor $\nabla:\M_1\to\A$ which maps $M\rarr{f}N$ to $\nu_N^{-1}(f(\xi_M))$ is a skew monoidal equivalence.
\end{lem}
\begin{proof}
Since all objects of $\M_1$ are isomorphic, $\nabla$ is clearly an equivalence. The only question is strong skew monoidality of $\nabla$. Omitting the details of a straightforward calculation we remark that the isomorphism $\nabla(M)\smpc \nabla(N)\iso \nabla(M\smpc N)$ is given by the unique element $\alpha_{M,N}\in A$ such that
\[
\xi_{M\smpc N}\cdot\alpha_{M,N}=[\xi_M,\xi_N,\mu]
\]
and choosing $\xi_R=1_A$, hence $\nu_R=\id_A$, the $\nabla$ can be made strictly normal.
\end{proof}
\begin{cor}
Every SMM is equivalent, as skew monoidal categories, to a full skew monoidal subcategory of a closed skew monoidal category.
\end{cor}
\begin{rmk}
Although all objects of $\M_1$ are isomorphic to the skew monoidal unit, the components of $\gamma$, $\eta$ and $\eps$ may have non-invertible components within $\M_1$. In fact if any one of the components of $\eta$ or $\eps$ is invertible then all of them are and $\A$ is a trivial SMM; see Lemma \ref{lem: mimosa}.
\end{rmk}

Before discussing the closed structure of SMM-s we prove an embedding theorem relating SMM-s to SMC-s.
In the following Theorem \textit{skew monoidal embedding} means a fully faithful strong skew monoidal functor.
\begin{thm} \label{thm: emb}
Let $\bra\M,\smpc,R,\gamma,\eta,\eps\ket$ be a skew monoidal category (SMC). Then for the existence of a SMM $\A$ and a skew monoidal embedding
\[
E: A\to\M
\]
it is sufficient and necessary that $R\smpc R\cong R$.
\end{thm}
\begin{proof}
Let $E:\A\to\M$ be a SM embedding of a SMM $\A$. Let $R'$ be the image of the single object of $\A$. Then $R\cong R'$ and $R'\smpc R'\cong R'$ by strong skew monoidality of $E$. Hence $R\smpc R\cong R$.

Assuming that $\kappa:R\smpc R\iso R$ is an isomorphism we can define a skew monoidal structure on the endomorphism monoid $A=\M(R,R)$ by setting
\begin{align*}
a\smp b&:=\kappa\ci(a\smpc b)\ci\kappa^{-1},\qquad a,b\in A\\
\gamma&:=\kappa\ci(\kappa\smpc R)\ci\gamma_{R,R,R}\ci(R\smpc\kappa^{-1})\ci\kappa^{-1}\\
\eta&:=\kappa\ci\eta_R\\
\eps&:=\eps_R\ci\kappa^{-1}
\end{align*}
Therefore $\A=\bra A,\smp,\gamma,\eta,\eps\ket$ is a skew monoidal monoid and
\[
E:\A\to\M,\qquad E(a)=(R\rarr{a}R)
\]
is a skew monoidal embedding the strong skew monoidal structure of which is given by $\kappa$ and by the identity arrow $R\rarr{=}R$. 
\end{proof}

As a skew monoidal category $\A$ being (left or/and right) closed is equivalent, by Lemma \ref{lem: M_1}, to $\M_1$ having this property. But $\M_1$ is a full skew monoidal subcategory of the closed $\M$ therefore any type of closedness of $\M_1$ is the question of whether the corresponding internal hom functors map the unit object $R$ into $\M_1$. Since
\begin{align*}
M\smp\under&\cong \under \oT(M\oS G)\cong\under\oT G\cong\under\oA (\,_{T(A)}A_A)\\
\under\smp M&\cong\under\oS(M\oT G)\cong\under\oS G\cong \under\oA(\,_{Q(A)}A_A)
\end{align*}
for all object $M\in\M_1$, the right and left internal homs are 
\begin{align}
[M,N]^r&\cong N_{T(A)}\\
[M,N]^\ell&\cong N_{Q(A)}\,.
\end{align}
respectively, where, e.g. $N_{T(A)}$ denotes the set $N$ equipped with right $A$-action $\bra n,a\ket\mapsto n\cdot T(a)$. 
This leads to the following result.
\begin{pro}
For a SMM $\A$ the following conditions are equivalent:
\begin{enumerate}
\item $\A$ is right closed, i.e., the $T:A\to A$ has a right adjoint in the 2-category $\Mon$.
\item The right internal hom $[R,R]^r\cong A_{T(A)}$ belongs to $\M_1$, i.e., $A$ is rank 1 free as right $A$-module via $T$.
\item $F_A\equiv \,_{S(A)}F$ is rank 1 free.
\end{enumerate}
Also, the following conditions are equivalent:
\begin{enumerate}
\item $\A$ is left closed, i.e., the $Q:A\to A$ has a right adjoint in the 2-category $\Mon$.
\item The left internal hom $[R,R]^\ell\cong A_{Q(A)}$ belongs to $\M_1$, i.e., $A$ is rank 1 free as right $A$-module via $Q$.
\item $G_{TQ(A)}$ is a rank 1 free right $A$-module.
\end{enumerate}
\end{pro}

Comparing left and right closedness, especially property (iii), we see that closedness is utterly asymmetic.
In a closed SMM the bialgebroid $F$ is both source-regular and target-regular but $G$ is not target-regular.
In order to get a symmetric notion let us say that $\A$ is `biclosed' if it is right closed and `left coclosed', i.e., require that 
$T$ has a right adjoint and $Q$ has a left adjoint. Then a little inspection shows that $\A$ is `biclosed' precisely when
both $F$ and $G$ are target-regular.

So far, in this Section, we have been working within the skew monoidal category $\M=\Set_A$ associated to the bialgebroid $G$.
There is another construction associated to a bialgebroid, the bimonad or opmonoidal monad acting on the bimodule category $_A\Set_A$. It is this latter language on which Hopf algebroids \cite{Sch} and Hopf monads \cite{BLV} can be defined. 

Identifying $_A\Set_A$ with $\Set_E$, where $E:=A^\op\x A$, the opmonoidal monad is given by tensoring with the $E$-monoid $G$,
\begin{align*}
\hat T X\ &:=\ X\oE G\\
\hat\mu_X&: (x\oE g)\oE h\ \mapsto\ x\oE gh\\
\hat\eta_X&: x\ \mapsto\ x\oE 1_G
\end{align*}
and the opmonoidal structure is 
\begin{alignat*}{2}
&\hat T_{X,Y}\ :&\quad (X\oA Y)\oE G&\to(X\oE G)\oA(Y\oE G)\\
&&\quad (x\ot y)\ot g&\mapsto (x\ot g\oneT)\ot(y\ot g\twoT)\\
&\hat T_0\ :&\quad A\oE G&\to A\\
&&\quad a\ot g&\mapsto \pi(T(a)g)\,.
\end{alignat*}
The connection between the opmonoidal monad and the canonical monad $R\smpc\under$ on $\Set_A$ is the isomorphism of their Eilenberg-Moore categories over the forgetful functor $_A\Set_A\to\Set_A$. But $R\smpc\under$ is lacking any (op)monoidal structure and such properties as a bialgebroid being a Hopf algebroid can only be formulated in terms of $\hat T$.  
The left and right fusion morphisms associated to $\hat T$ are
\begin{align}
\label{Hleft}
H^\ell_{X,Y}&=(\hat T X\oA \hat\mu_Y)\ci\hat T_{X,\hat TY}\ :\ \hat T(X\oA \hat TY)\to\hat TX\oA\hat TY\\
&:\ (x\oA(y\oE g))\oE h\ \mapsto\ (x\oE h\oneT)\oA(y\oE gh\twoT)\\
\label{Hright}
H^r_{X,Y}&=(\hat\mu_X\oA\hat T Y)\ci\hat T_{\hat T X,Y}\ :\ \hat T(\hat TX\oA Y)\to\hat TX\oA \hat TY\\
&:\ ((x\oE g)\oA y)\oE h\ \mapsto\ (x\oE gh\oneT)\oA(y\oE h\twoT)
\end{align}
and we know from  \cite[Theorem 3.6]{BLV} that $\hat T$ is left (right) Hopf, i.e., $H^\ell$ (resp. $H^r$) is invertible, precisely when
$(\,_A\Set_A)^{\hat T}$ is left (right) closed and the forgetful functor to $_A\Set_A$ preserves the left (right) internal hom. 
If we try to translate this Theorem to the skew monoidal language we again run into the problem that left and right closed structures behave differently.
\begin{lem}
For a skew monoidal monoid $\A=\bra A,T,Q,\gamma,\eta,\eps\ket$ the following conditions are equivalent:
\begin{enumerate}
\item $G$, the right $A$-bialgebroid associated to $\A$, is left Hopf, i.e., the opmonoidal monad $\hat T=\,\under\oE G$ is a left Hopf monad, 
\item the canonical map
\[
\can^G:\ G\am{T(A)}G\to G\oA G,\quad g\am{T(A)}h\mapsto h\oneT\oA gh\twoT
\]
is a bijection,
\item $\gamma$ is invertible.
\item the canonical map
\[
\can_F:\ F\am{Q(A)}F\to F\oA F,\quad f\am{T(A)}f'\mapsto f\oneB f'\oA f\twoB
\]
is a bijection,
\item $F$, the left $A$-bialgebroid associated to $\A$, is right Hopf, i.e., the opmonoidal monad \newline
{$F\am{A\x A^\op}\under$} is a right Hopf monad.
\end{enumerate}
\end{lem}
\begin{proof}
(i)$\Leftrightarrow$(ii) is obvious from formula (\ref{Hleft}). 

(ii)$\Leftrightarrow$(iii) Composing the canonical map with the isomorphisms
\begin{alignat*}{2}
\tau: G\oA G&\iso G,&\qquad  g_1\oA g_2&\mapsto g_2 S(g_1\eta)\\
\sigma:G\am{T(A)}G&\iso G,&\qquad g_1\am{T(A)}g_2&\mapsto T(g_1\eta)g_2
\end{alignat*}
we obtain a map
\[
c^G:G\to G,\qquad g\mapsto \mu T(\mu\delta g\eta)
\]
such that $\tau\ci \can^G=c^G\ci\sigma$.
Using the bijection $K:A\iso G$, $a\mapsto \mu T(a)$, we can write $c^G(g)=K(\gamma K^{-1}(g))$ from which the statement follows.

By dualizing the above arguments one obtains (iii)$\Leftrightarrow$(iv) and (iv)$\Leftrightarrow$(v).
\end{proof}

SMM-s with the above five properties could be called one-sided Hopf SMM although we cannot say whether left or right due to the rival conditions (i) and (v). Notice also that these one-sided Hopf properties do not require the skew monoidal structure of $\A$ to
be closed. 

The other Hopf property, i.e., invertibility of $H^r$ for $G$ or invertibility of $H^\ell$ for $F$, is equivalent to the maps
\begin{alignat*}{2}
{\can'}^G: G\am{S(A)}G&\to G\oA G,&\qquad g\am{S(A)}h&\mapsto gh\oneT\oA h\twoT\\
\can'_F: F\am{S(A)}F&\to F\oA F,&\qquad f\am{S(A)}f'&\mapsto f\oneB\oA f\twoB f'
\end{alignat*}
being bijective. Unfortunately these conditions are not equivalent to invertibility of some simple combination of the skew monoidal structure maps. However, if we assume the $\A$ is left coclosed, or, just for symmetry, that it is biclosed, then there is 
a simple characterization in terms of a \emph{left} skew monoidal structure. (So far, in this paper we used `skew monoidal' for a right skew monoidal structure.) If $Q$ has a left adjoint, let's say $P$ with unit $i:\id_A\to QP$ and counit $e:PQ\to\id_A$,
then a left skew monoidal structure on the monoid $A$ can be defined by
\begin{align*}
a\smp' b&:=eP(S(b)ia)\,,\qquad a,b\in A\\
\gamma'&:=eP(\mu T(i)e P(\delta i))\\
\eta'&:=eP(\eta)\\
\eps'&:=\eps i\,.
\end{align*}
This structure can be obtained via a contravariant embedding (similarly to Theorem \ref{thm: emb}) into the right skew monoidal 
category $_A\Set$ induced by the right $A^\op$-bialgebroid $G^\coop$.
\begin{lem}
For a biclosed skew monoidal monoid $\A=\bra A,T,Q,\gamma,\eta,\eps\ket$ the following conditions are equivalent:
\begin{enumerate}
\item $G$, the right $A$-bialgebroid associated to $\A$, is right Hopf.
\item ${\can'}^G$ is a bijection.
\item $\gamma'$ is invertible.
\end{enumerate}
\end{lem}

Concluding the paper we mention that the existence of non-trivial skew monoidal monoids is an open problem. It seems that the SMM structure is very sensitive to the addition of new axioms and can easily collapse into something trivial. 
Let us make this precise.

\begin{defi}
A skew monoidal monoid $\bra A,T,Q,\mu,\eta,\delta,\eps\ket$ is called trivial if
$A$ is commutative, $T=Q=\id_A$, $\eps$ is invertible and $\mu=\eps$, $\delta=\eta=\eps^{-1}$.
\end{defi}

\begin{lem}\label{lem: mimosa}
For a skew monoidal monoid $\A$ any one of the following conditions alone implies that $\A$ is a trivial skew monoidal monoid:
\begin{enumerate}
\item $A$ is finite.
\item $A$ is a cancellative monoid.
\item $A$ is commutative.
\item $T$ is an automorphism.
\item $Q$ is an automorphism.
\item $\eta$ is invertible.
\item $\mu$ is invertible.
\item $\delta$ is invertible.
\item $\eps$ is invertible.
\item $\mu\hotobj\mu\cong\eps$ as objects of $\A^T$.
\end{enumerate}
\end{lem}
Notice that invertibility of $\gamma$ does not occur among these fatal conditions, thereby granting Hopf skew monoidal monoids the chance of being non-trivial.

\end{document}